\documentclass[12pt,leqno]{amsart}
\usepackage[latin9]{inputenc}
\usepackage{verbatim}
\usepackage{amsthm}
\usepackage{amstext}
\usepackage{amssymb}
\usepackage[dvipdfmx]{graphicx}
\usepackage{xcolor}
\usepackage{soul}
\makeatletter
\numberwithin{equation}{section}
\numberwithin{figure}{section}

\usepackage{amsthm}\usepackage{amscd}\usepackage{bm}

\theoremstyle{plain}
\newtheorem{main theorem}{Main Theorem}
\newtheorem{theorem}{Theorem}[section]
\newtheorem{lemma}[theorem]{Lemma}

\newtheorem{proposition}[theorem]{Proposition}

\theoremstyle{definition}

\newtheorem{remark}[theorem]{Remark}
\newtheorem{example}[theorem]{Example}

\newtheorem{problem}[theorem]{Problem}

\newcommand{\mdim}{\mathrm{mdim}}
\newcommand{\diam}{\mathrm{diam}}
\newcommand{\widim}{\mathrm{Widim}}

\newcommand{\mesh}{\mathrm{mesh}}

\makeatother

\begin{document}

\title[Mean dimension of full shifts]{Mean dimension of full shifts}

\author{Masaki Tsukamoto}



\date{\today}

\thanks{}

\maketitle

\begin{abstract}
Let $K$ be a finite dimensional compact metric space and $K^\mathbb{Z}$ the full shift on the alphabet $K$.
We prove that its mean dimension is given by $\dim K$ or $\dim K-1$ depending on the ``type'' of $K$.
We propose a problem which seems interesting from the view point of infinite dimensional topology.
\end{abstract}

\section{Introduction}  \label{section: introduction}

Let $X$ be a compact metric space and $T:X\to X$ a homeomorphism.
We call $(X,T)$ a \textbf{dynamical system}.
The most basic invariant of dynamical systems is topological entropy.
The topological entropy $h_{\mathrm{top}}(X,T)$ evaluates how many bits per iterate we need to 
describe an orbit in $(X,T)$.
At the end of the last century Gromov \cite{Gromov} introduced a new topological invariant of dynamical systems 
called mean dimension, denoted by $\mdim(X,T)$.
It evaluates how many \textit{parameters} per iterate we need to describe an orbit in $(X,T)$.
Mean dimension has applications to topological dynamics which cannot be touched within the framework of entropy theory 
\cite{Lindenstrauss--Weiss, Lindenstrauss, Gutman 1, Gutman--Lindenstrauss--Tsukamoto, Gutman--Tsukamoto minimal, 
Meyerovitch--Tsukamoto}.

Let $[0,1]^D$ be the $D$-dimensional cube and consider the shift 
$\sigma: \left([0,1]^D\right)^\mathbb{Z}  \to \left([0,1]^D\right)^\mathbb{Z}$
defined by 
\[ \sigma\left((x_n)_{n\in \mathbb{Z}}\right) = \left(x_{n+1}\right)_{n\in\mathbb{Z}}, \quad 
    (x_n\in [0,1]^D). \]
Then its mean dimension is given by \cite[Proposition 3.3]{Lindenstrauss--Weiss}
\begin{equation} \label{eq: mean dimension of Hilbert cube}
   \mdim\left( \left([0,1]^D\right)^\mathbb{Z}, \sigma\right) = D.
\end{equation}
This is the simplest calculation of mean dimension.

The main problem we study here is to extend the calculation (\ref{eq: mean dimension of Hilbert cube}) to arbitrary full shifts.
Let $K$ be a compact metric space and $\sigma: K^\mathbb{Z}\to K^\mathbb{Z}$ the shift on the alphabet $K$.
It is generally true that \cite[Proposition 3.1]{Lindenstrauss--Weiss}
\begin{equation} \label{eq: trivial upper bound on mean dimension}
    \mdim\left(K^\mathbb{Z}, \sigma\right) \leq \dim K, 
\end{equation}    
where the right-hand side is the topological dimension of $K$.
From (\ref{eq: mean dimension of Hilbert cube}) we see that the equality 
$\mdim\left(K^\mathbb{Z}, \sigma\right) = \dim K$ holds if $K$ is a nice space (e.g. topological manifold, polyhedron).
But in general the equality does not hold in (\ref{eq: trivial upper bound on mean dimension}).
It is known (Boltyanski\v{i} \cite{Boltyanskii}; Nagami \cite[\S 40]{Nagami} might be a more easily accessible literature)
that there exists a compact metric space $B$ (the \textit{Boltyanski\v{i} surface})
satisfying 
\begin{equation} \label{eq: Boltyanskii}
  \dim B = 2, \quad \dim B^2 = 3.
\end{equation}
Then it is easy to see $\mdim\left(B^\mathbb{Z}, \sigma\right) \leq 3/2 < 2 = \dim B$.

The full shift $K^\mathbb{Z}$ is certainly the most basic example in mean dimension theory.
So apparently it looks strange that nobody has ever carried out the calculation of its mean dimension.
Probably this is because the above Boltyanski\v{i} surface $B$ gave researchers an impression that 
we cannot expect a clear result.
However this impression turns out to be wrong.
We prove a fairly satisfactory answer below (except for one remaining question; see Problem \ref{problem: infinite dimension}).

Following Dranishnikov \cite[Section 3]{Dranishnikov},
we introduce the notion of ``type'' for finite dimensional compact metric spaces.
Let $K$ be a finite dimensional compact metric space.
It is known that (\cite[Theorem 3.16]{Dranishnikov})
\begin{equation}  \label{eq: basic/exceptional}
    2\dim K-1 \leq \dim K^2 \leq 2\dim K. 
\end{equation}    
Since topological dimension is an integer, we have either 
\[  \dim K^2 = 2\dim K, \text{ or } \dim K^2 = 2\dim K-1. \]
$K$ is said to be \textbf{of basic type} if $\dim K^2 = 2\dim K$.
Otherwise (namely, if $\dim K^2 = 2\dim K-1$), $K$ is said to be \textbf{of exceptional type}.
For example, the Boltyanski\v{i} surface $B$ has exceptional type.
The following is our main result.

\begin{theorem}  \label{main theorem}
Let $K$ be a finite dimensional compact metric space and $\sigma: K^\mathbb{Z}\to K^\mathbb{Z}$
the full shift on the alphabet $K$.
   \begin{enumerate}
      \item If $K$ has basic type then $\mdim\left(K^\mathbb{Z},\sigma\right) = \dim K$.
      \item If $K$ has exceptional type then $\mdim\left(K^\mathbb{Z},\sigma\right) = \dim K-1$.
   \end{enumerate}
\end{theorem}

\begin{example}
   \begin{enumerate}
      \item The  Boltyanski\v{i} surface $B$ in (\ref{eq: Boltyanskii}) satisfies $\mdim\left(B^\mathbb{Z},\sigma\right) = 1$.
      
      \item Pontryagin \cite{Pontryagin} constructed compact metric spaces (the \textit{Pontryagin surfaces}) $P$ and $Q$ satisfying 
      \[  \dim P = \dim Q =2 , \quad \dim P\times Q = 3. \]
       and that $P$, $Q$ and $P\times Q$
       are all of basic type. (This can be checked by the calculations of cohomological dimensions given in \cite[\S\S1-3]{Dranishnikov}.)
       Let $\sigma_1:P^\mathbb{Z}\to P^\mathbb{Z}$ and $\sigma_2:Q^\mathbb{Z}\to Q^\mathbb{Z}$ be the shifts.
       Then 
       \[  \mdim\left(P^\mathbb{Z}, \sigma_1\right) = \mdim\left(Q^\mathbb{Z},\sigma_2\right) = 2, \quad 
           \mdim\left(P^\mathbb{Z}\times Q^\mathbb{Z}, \sigma_1\times \sigma_2\right) = 3. \]
       In particular, $\mdim\left(P^\mathbb{Z}\times Q^\mathbb{Z}, \sigma_1\times \sigma_2\right)
       < \mdim\left(P^\mathbb{Z},\sigma_1\right) + \mdim\left(Q^\mathbb{Z},\sigma_2\right)$.
       This is the first example where the inequality 
       \[  \mdim(X\times Y, T\times S) \leq \mdim(X,T) + \mdim(Y,S),   \]
       which holds for any dynamical systems $(X,T)$ and $(Y,S)$, becomes strict.
    \end{enumerate}
\end{example}      

Some readers might wonder why we exclude the case of \textit{infinite dimensional} $K$ in Theorem \ref{main theorem}.
Indeed this is \textit{the} open problem:

\begin{problem}   \label{problem: infinite dimension}
Let $K$ be an infinite dimensional compact metric space.
Is the mean dimension $\mdim\left(K^\mathbb{Z},\sigma\right)$ infinite?
\end{problem}

The difficulty of this problem comes from the following two remarkable phenomena in infinite dimensional topology.

\begin{enumerate}
   \item There exists an infinite dimension compact metric space $K$ containing no \textit{intermediate dimensional subspaces}
   \cite{Henderson, Rubin--Shori--Walsh, Walsh}.
   Namely every closed subset $A\subset K$ satisfies either $\dim A = 0$ or $\dim A = \infty$.

   \item There exists an infinite dimensional compact metric space $K$ which \textit{cohomologically looks like a surface}
   \cite{Dydak--Walsh}.
   Namely for every closed subset $A\subset K$ and $n\geq 3$ the \v{C}ech cohomology group $\check{H}^n(K,A)$ vanishes.
\end{enumerate}

These two difficulties are genuinely infinite dimensional phenomena.
The difficulty (1) implies that we cannot reduce the problem to a finite dimensional case, and the difficulty (2) implies that 
the ordinary cohomology theory is insufficient to solve the problem.
I hope that this paper will stimulate an expert of infinite dimensional topology to solve the above problem.
Then we will get a complete understanding of the mean dimension of full shifts.

\begin{remark}
If $K$ is a positive dimensional compact metric space (possibly infinite dimensional) then 
$\mdim\left(K^\mathbb{Z}, \sigma\right)\geq 1$. This can be proved by the cohomological method used in the proof of 
Theorem \ref{main theorem}.
But it seems difficult to improve this by the same method.
\end{remark}

The purpose of this paper is to prove Theorem \ref{main theorem}.
In \S \ref{section: preliminaries} we review basics of dimension/mean dimension/cohomological dimension theories.
We prove Theorem \ref{main theorem} in \S \ref{section: proof of main theorem}.

\section{Preliminaries}  \label{section: preliminaries}

\subsection{Dimension and mean dimension} \label{subsection: dimension and mean dimension}

Here we review basics of topological dimension and mean dimension
\cite{Engelking, Gromov, Lindenstrauss--Weiss, Lindenstrauss}.

Let $(X,d)$ be a compact metric space.
For an open cover $\alpha$ of $X$ we define the \textbf{order} $\mathrm{ord}(\alpha)$ as the maximum $n\geq 0$ such that 
there exist pairwise distinct open sets $U_0,\dots, U_n\in \alpha$ with $U_0\cap \dots \cap U_n \neq \emptyset$.
An open cover $\beta$ of $X$ is said to be a \textbf{refinement} of $\alpha$ (denoted by $\beta \geq \alpha$)
if for every $V\in \beta$ there exists $U\in \alpha$ satisfying $V\subset U$.
We define the \textbf{degree} $D(\alpha)$ of $\alpha$ as the minimum of $\mathrm{ord}(\beta)$ over all refinements $\beta$
of $\alpha$.
The topological dimension of $X$ is given by 
\[ \dim X = \sup_{\text{$\alpha$: open cover of $X$}} D(\alpha). \]
For an open cover $\alpha$ of $X$ we define $\mathrm{mesh}(\alpha, d)$ as the supremum of the diameter of $U$ over $U\in \alpha$.
For $\varepsilon >0$ we define the \textbf{$\varepsilon$-width dimension} $\widim_\varepsilon (X,d)$ as the minimum of $D(\alpha)$
over all open covers $\alpha$ of $X$ satisfying $\mesh(\alpha, d) < \varepsilon$. 
Then $\dim X$ can be also written as
\[ \dim X = \lim_{\varepsilon \to 0} \widim_\varepsilon (X,d). \]

Let $T:X\to X$ be a homeomorphism.
For each $n \geq 1$ we define a new distance $d_n$ on $X$ by 
\[ d_n(x,y) = \max_{0\leq i <n} d(T^i x, T^i y). \]
We define the mean dimension $\mdim(X,T)$ by 
\[ \mdim(X,T) = \lim_{\varepsilon \to 0} \left(\lim_{n\to \infty} \frac{\widim_\varepsilon(X,d_n)}{n}\right). \]
The limit with respect to $n$ exists because $\widim_\varepsilon(X,d_n)$ is subadditive in $n$.
The mean dimension $\mdim(X,T)$ is a topological invariant of $(X,T)$, namely it is independent of the choice of $d$
compatible with the underlying topology.
The following lemma will be used later \cite[Proposition 2.7]{Lindenstrauss--Weiss}.

\begin{lemma}  \label{lemma: amplification}
$\mdim(X, T^n)= n\cdot \mdim(X,T)$.
\end{lemma}

\subsection{Cohomological dimension}  \label{subsection: cohomological dimension}

We review cohomological dimension theory here.
An excellent reference is the survey of Dranishnikov \cite{Dranishnikov}.
We need only basic results, which are all covered by \S\S 1-3 of \cite{Dranishnikov}.

Cohomological dimension theory uses the \v{C}ech cohmology, which is different from the
more standard singular cohomology.
So we first review its definition, following \cite[pp. 358-360]{Spanier} and \cite[\S 34]{Nagami}. 
Let $K$ be a compact metric space (here we switch our notation from $X$ to $K$ because we will consider 
$X= K^\mathbb{Z}$ later) and $G$ an Abelian group.
Let $A\subset K$ be a closed subset.
Let $\alpha$ be an open cover of $K$. (Since $K$ is compact, it is enough to consider the case of 
finite open covers $\alpha$.)
Set $\alpha|_A = \{U\in \alpha|\, U\cap A\neq \emptyset \}$.
We denote by $N(\alpha)$ and $N(\alpha|_A)$ the nerve complexes of $\alpha$ and $(U\cap A)_{U\in \alpha|_A}$ respectively.
We naturally consider $N(\alpha|_A)$ as a subcomplex of $N(\alpha)$.
Let 
\begin{equation} \label{eq: cohomology of nerves}
     H^*\left(N(\alpha), N(\alpha|_A);G\right) 
\end{equation}     
be the (say, singular\footnote{Here we only need to consider finite simplicial complexes.
So any cohomology theory provides the same result.}) cohomology group of the pair $\left(N(\alpha), N(\alpha|_A)\right)$
with the coefficient group $G$.
If $\beta\geq \alpha$ is a refinement, we can define a simplicial map from 
$\left(N(\beta), N(\beta|_A)\right)$ to $\left(N(\alpha), N(\alpha|_A)\right)$.
Although this map itself is not canonically defined, the induced homomorphism 
\[  H^*\left(N(\alpha), N(\alpha|_A); G \right) \to H^*\left(N(\beta), N(\beta|_A); G\right) \]
is canonically defined.
So (\ref{eq: cohomology of nerves}) forms a directed system where $\alpha$ runs over open covers of $K$.
We define the \textbf{\v{C}ech cohomology group} $\check{H}^*(K, A; G)$ as the direct limit
\[  \check{H}^*(K,A;G) = \varinjlim_{\alpha} H^*\left(N(\alpha), N(\alpha|_A); G\right). \]
The following trivial lemma will be crucial later.

\begin{lemma} \label{lemma: effective cohomology}
If the natural map $H^q\left(N(\alpha), N(\alpha|_A); G\right) \to \check{H}^q(K,A;G)$ is nonzero for some $q\geq 0$ then
$D(\alpha) \geq q$.
\end{lemma}

\begin{proof}
Suppose $D(\alpha) < q$. Then there exists $\beta\geq \alpha$ with $\mathrm{ord}(\beta) < q$.
The nerve complex $N(\beta)$ has dimension smaller than $q$.
So $H^q\left(N(\beta), N(\beta|_A); G\right) =0$.
The map $H^q\left(N(\alpha), N(\alpha|_A); G\right) \to \check{H}^q(K,A;G)$ factors into
\[ H^q\left(N(\alpha), N(\alpha|_A); G\right) \to H^q\left(N(\beta), N(\beta|_A); G\right) \to   \check{H}^q(K,A;G). \]
\end{proof}

We define the \textbf{cohomological dimension} $\dim_G K$ as the supremum of $q\geq 0$
satisfying $\check{H}^q(K, A; G) \neq 0$ for some closed subset $A\subset K$.
An immediate consequence of Lemma \ref{lemma: effective cohomology}
is that $\dim_G K \leq \dim K$.
At least for finite dimensional\footnote{The fundamental difficulty of Problem \ref{problem: infinite dimension} 
comes from the failure of Theorem \ref{theorem: Alexandroff} for infinite dimensional spaces.}
 $K$, the cohomology has an enough information to determine the topological dimension
\cite[Theorem 1.4]{Dranishnikov}:

\begin{theorem}[Alexandroff]  \label{theorem: Alexandroff}
 If $K$ is finite dimensional then $\dim K = \dim_\mathbb{Z} K$.
\end{theorem}

The following result is given in \cite[Lemma 2.9]{Dranishnikov}.

\begin{theorem}  \label{theorem: finding good coefficient}
There exists a field $F$ (depending on $K$) satisfying $\dim_F K \geq \dim_\mathbb{Z} K-1$.
Indeed we can take $F = \mathbb{Q}$ or $\mathbb{Z}/p\mathbb{Z}$ for some prime number $p$.
\end{theorem}

\begin{proof}
This follows from the universal coefficient formula (\cite[p. 336]{Spanier}, \cite[\S 39-4]{Nagami}):
\begin{equation}  \label{eq: universal coefficient formula}  
    0 \rightarrow     \check{H}^n(K,A) \otimes F \rightarrow \check{H}^n(K,A; F) \rightarrow 
     \check{H}^{n+1}(K,A) * F \rightarrow 0.   
\end{equation}   
Here $*$ is the torsion product \cite[p.220]{Spanier}.  
Take any $q\leq \dim_\mathbb{Z} K$ with $\check{H}^q(K,A) \neq \emptyset$ for some closed subset $A \subset K$.   
If $\check{H}^q(K,A)$ contains a non-torsion element u (i.e. $m u \neq 0$ for all $m \geq 1$) then 
$\check{H}^q(K,A)\otimes \mathbb{Q}\neq 0$.
It follows from (\ref{eq: universal coefficient formula}) with $F=\mathbb{Q}$ and $n=q$ that $\check{H}^q(K,A;\mathbb{Q}) \neq 0$
and hence $\dim_\mathbb{Q} K \geq q$.
If $\check{H}^q(K,A)$ contains a torsion element $u\neq 0$ (let $m> 1$ be the order of $u$) then 
$\check{H}^q(K,A) * (\mathbb{Z}/p \mathbb{Z}) \neq 0$ for prime divisors $p$ of $m$ 
because 
\[  0 \rightarrow (\mathbb{Z}/m\mathbb{Z}) * (\mathbb{Z}/p\mathbb{Z}) \rightarrow \check{H}^q(K,A) * (\mathbb{Z}/p \mathbb{Z}), 
    \quad   (\mathbb{Z}/m\mathbb{Z}) * (\mathbb{Z}/p\mathbb{Z}) \cong \mathbb{Z}/p \mathbb{Z}. \]
It follows from (\ref{eq: universal coefficient formula}) with $F=\mathbb{Z}/p\mathbb{Z}$ and $n=q-1$ that 
$\check{H}^{q-1} \left(K,A;\mathbb{Z}/p\mathbb{Z}\right) \neq 0$ and hence $\dim_{\mathbb{Z}/p\mathbb{Z}} K \geq q-1$.
\end{proof}

The next theorem shows radically different behaviors of compact metric spaces of basic/exceptional types.
The proof is given in \cite[Theorem 3.16]{Dranishnikov}. 
See also the last two paragraphs of \S 3 of \cite{Dranishnikov}
where the basic/exceptional dichotomy is explained.

\begin{theorem}  \label{theorem: basic/exceptional}
If $K$ is finite dimensional then
\[ \dim K^n = \begin{cases}  n \dim K  & (\text{if $K$ has basic type}) \\
                                      n \dim K -n +1 & (\text{if $K$ has exceptional type}).
                   \end{cases}                      
\]
\end{theorem}

\begin{proof}[Sketch of the proof]
We sketch the proof of the following weaker statement, 
which is enough\footnote{Strictly speaking, this weaker statement misses a point;
it does not yield (\ref{eq: basic/exceptional}). 
It gives only
$2\dim K -2 \leq \dim K^2 \leq 2\dim K$.
So we cannot define the basic/exceptional dichotomy from this. 
But if we define that \textit{$K$ has exceptional type if $K$ does not have basic type} then everything becomes logically consistent.

We can prove $\dim K^2 \geq 2\dim K-1$ by investigating
 the K\"{u}nneth formula (\cite[pp. 359-360]{Spanier}, \cite[\S 41]{Nagami})
 \[  0 \rightarrow \left(\check{H}^*(K,A)\otimes \check{H}^*(K,A)\right)^q \rightarrow 
      \check{H}^q\left((K,A)\times (K,A)\right) \rightarrow 
      \left(\check{H}^*(K,A) * \check{H}^*(K,A)\right)^{q+1} \rightarrow 0. \] 
      We use this with $q=2\dim K$ or $q=2\dim K -1$, depending on whether $\check{H}^{\dim K}(K,A)$ contains a non-torsion element or not.}
 for the proof of Theorem \ref{main theorem}.
(Recall: We defined that $K$ has basic type if $\dim K^2 = 2\dim K$.)
\begin{equation*}
   \begin{split}
   & \dim K^n = n \dim K \quad (\text{if $K$ has basic type}),  \\
   n \dim K -n \leq & \dim K^n \leq n \dim K -n +1 \quad (\text{if $K$ does not have basic type}). 
  \end{split} 
\end{equation*} 

We use the following fact: If $F$ is a field then $\dim_F K^n = n \dim_F K$ (\cite[Proposition 3.3]{Dranishnikov}).
The inequality $\dim_F K^n \geq n \dim_F K$ directly follows from  the K\"{u}nneth formula $\check{H}^*((K,A)^n; F) \cong \check{H}^*(K,A;F)^{\otimes n}$
(\cite[pp. 359-360]{Spanier}, \cite[\S 41]{Nagami}).
The proof of the reverse inequality $\dim_F K^n \leq n \dim_F K$ is a bit more involved (based on the Mayer--Vietoris exact sequence 
of compact-supported cohomology) and we omit it here.

First suppose that there exists a field $F$ satisfying $\dim_F K = \dim K$.
Then $\dim_F K^n = n \dim_F K = n \dim K$. 
Hence $\dim K^n \geq \dim_F K^n = n \dim K$.
Since the reverse inequality $\dim K^n \leq n \dim K$ is always true, we get $\dim K^n = n \dim K$.
In particular $K$ has basic type.

Next suppose that all fields $F$ satisfy $\dim_F K <\dim K$.
Take a field $F$ satisfying $\dim_F K = \dim K -1$ (Theorem \ref{theorem: finding good coefficient}).
Since $\dim K^n \geq \dim_F K^n = n \dim_F K = n(\dim K -1)$, we get $\dim K^n \geq n \dim K -n$.

Let $n\geq 2$.
Take a field $F'$ (depending on $n$) satisfying $\dim_{F'} K^n \geq \dim K^n -1$ (Theorem \ref{theorem: finding good coefficient}).
We must have $\dim_{F'} K \leq \dim K-1$.
Hence $\dim K^n \leq \dim_{F'} K^n + 1 = n \dim_{F'} K +1 \leq  n \dim K -n+1$.
In particular $K$ does not have basic type.
\end{proof}

\section{Proof of Theorem \ref{main theorem}}    \label{section: proof of main theorem}.

Throughout this section we assume that $K$ is a finite dimensional compact metric space
with a distance function $\rho$. 
Let $n \geq 1$. We define a distance $\rho_n$ on $K^n$ by 
\[ \rho_n \left((x_0,\dots,x_{n-1}), (y_0,\dots, y_{n-1})\right) = \max_{0\leq i <n} \rho(x_i,y_i). \]

\begin{lemma}  \label{lemma: uniform lower bound on widim}
There exists $\delta>0$ such that for every $n \geq 1$ 
\[  \widim_\delta\left(K^n, \rho_n\right) \geq n(\dim K-1). \]
Here the point is that $\delta$ is independent of $n$.
\end{lemma}

\begin{proof}
By Theorems \ref{theorem: Alexandroff} and \ref{theorem: finding good coefficient}
we can find a field $F$ satisfying $\dim_F K \geq \dim K-1$.
Set $q = \dim_F K$.
There exists a closed subset $A\subset K$ with $\check{H}^q(K,A;F) \neq 0$.
We take an open cover $\alpha$ of $K$ such that the natural map 
$H^q\left(N(\alpha), N(\alpha|_A); F\right) \to \check{H}^q(K,A;F)$ is nonzero.
Let $n\geq 1$ and define a closed subset $A_n\subset K^n$ by 
$(K^n, A_n) = (K,A)^n$, namely $A_n$ is the set of points $(x_0,\dots, x_{n-1})\in K^n$ satisfying 
$x_i\in A$ for some $i$.
We define an open cover $\alpha^n$ of $K^n$ by 
\[ \alpha^n = \{U_0\times \dots \times U_{n-1}|\, U_0,\dots, U_{n-1}\in \alpha\}. \]

We have the following commutative diagram:
\[
  \begin{CD}
     H^q\left(N(\alpha), N(\alpha|_A) ; F \right)^{\otimes n} @>>> H^{nq}\left(N(\alpha^n),N(\alpha^n|_{A_n});F\right) \\
     @VVV    @VVV \\
     \check{H}^q\left(K,A;F\right)^{\otimes n}   @>>>  \check{H}^{nq}(K^n, A_n; F)
  \end{CD}
\]
The horizontal arrows are injective (the K\"{u}nneth formula \cite[pp. 359-360]{Spanier}, \cite[\S 41]{Nagami}).
Since $F$ is a field and the map $H^q\left(N(\alpha), N(\alpha|_A); F\right) \to \check{H}^q(K,A;F)$ is nonzero,
the left vertical arrow is nonzero. 
So the right vertical arrow $H^{nq}\left(N(\alpha^n),N(\alpha^n|_{A_n});F\right) \to  \check{H}^{nq}(K^n, A_n; F)$
is nonzero.
It follows from Lemma \ref{lemma: effective cohomology} that $D(\alpha^n) \geq nq$.

Let $\delta>0$ be the Lebesgue number of $\alpha$;
namely if $V\subset K$ satisfies $\diam V < \delta$ then there exists $U\in \alpha$ with $V\subset U$.
If $\beta$ is an open cover of $K^n$ with $\mesh\left(\beta, \rho_n\right) < \delta$ then 
$\beta\geq \alpha^n$ and hence $\mathrm{ord}(\beta) \geq D(\alpha^n) \geq nq$.
This implies 
\[  \widim_\delta\left(K^n,\rho_n\right) \geq n q \geq n(\dim K-1). \]
\end{proof}

\begin{remark}
Gromov \cite[\S\S 2.6.1-2.6.2]{Gromov} proved a result very close to the above lemma.
Indeed I came up with the proof of Lemma \ref{lemma: uniform lower bound on widim}
when I tried to understand \S 2.6 of \cite{Gromov}, where he developed a cohomological approach to mean dimension 
in a quite broader perspective.
I recommend interested readers to look at his paper. It certainly contains (too) many unexplored issues.
\end{remark}

Let $\sigma: K^\mathbb{Z} \to K^\mathbb{Z}$ be the full shift on the alphabet $K$.

\begin{proposition}  \label{main proposition}
   $\mdim\left(K^\mathbb{Z},\sigma\right) \geq \dim K-1$.
\end{proposition}

\begin{proof}

We define a distance $d$ on $K^\mathbb{Z}$ by 
\[  d\left((x_i)_{i\in \mathbb{Z}}, (y_i)_{i\in \mathbb{Z}}\right) = \sum_{i\in \mathbb{Z}} 2^{-|i|} \rho(x_i,y_i). \]
Recall that we defined the distance $d_n$ on $K^\mathbb{Z}$ for each $n\geq 1$ by 
\[  d_n(x,y) = \max_{0\leq i <n} d(\sigma^i x, \sigma^i y). \]
Fix a point $p\in K$.
Let $n\geq 1$. We define a continuous map $f:K^n \to K^\mathbb{Z}$ by 
\[  f(x_0,\dots, x_{n-1})_i = \begin{cases} 
                                     x_i  & (0\leq i\leq n-1) \\
                                     p   &  (\text{otherwise}).
                                   \end{cases}  \]  
Then $\rho_n(x,y) \leq d_n\left(f(x), f(y)\right)$ and hence it follows that for $\varepsilon >0$
\[ \widim_\varepsilon \left(K^\mathbb{Z}, d_n\right) \geq \widim_\varepsilon \left(K^n, \rho_n\right). \]
We use Lemma \ref{lemma: uniform lower bound on widim} and get 
\[ \widim_\varepsilon \left(K^\mathbb{Z},d_n\right) \geq n (\dim K-1) \quad (0<\varepsilon \leq \delta). \]
Thus
\[  \mdim\left(K^\mathbb{Z},\sigma\right) = \lim_{\varepsilon \to 0}
    \left(\lim_{n\to \infty} \frac{\widim_\varepsilon \left(K^\mathbb{Z}, d_n\right)}{n} \right) \geq \dim K -1. \]
\end{proof}

Now we are ready to prove Theorem \ref{main theorem}.

\begin{proof}[Proof of Theorem \ref{main theorem}]
It follows from (\ref{eq: trivial upper bound on mean dimension}) and Proposition \ref{main proposition} that 
\[  \dim K -1 \leq \mdim\left(K^\mathbb{Z}, \sigma\right) \leq \dim K. \]
We amplify this by replacing $\left(K^\mathbb{Z},\sigma\right)$ with 
$\left(K^\mathbb{Z},\sigma^n\right)$, which is isomorphic to the full shift on the alphabet $K^n$.
Noting Lemma \ref{lemma: amplification}, we get 
\[ \frac{\dim K^n -1}{n} \leq \mdim\left(K^\mathbb{Z},\sigma\right) \leq \frac{\dim K^n}{n}. \]
It follows from Theorem \ref{theorem: basic/exceptional} that 
\[ \frac{\dim K^n}{n} = \begin{cases}   \dim K  & (\text{if $K$ has basic type}) \\
                                       \dim K -1 + \frac{1}{n} & (\text{if $K$ has exceptional type}).
                   \end{cases}                      
\]
Letting $n\to \infty$ we get the claim of the theorem.
\end{proof}

\vspace{0.5cm}

\address{ Masaki Tsukamoto \endgraf
Department of Mathematics, Kyoto University, Kyoto 606-8502, Japan}

\textit{E-mail address}: \texttt{tukamoto@math.kyoto-u.ac.jp}

\end{document}